 \documentclass[letterpaper, 10 pt, conference]{ieeeconf}

\usepackage{ulem} 
\usepackage{hyperref}
\usepackage{times,graphicx}
\usepackage{amsmath,amsfonts,stmaryrd,amssymb}

\usepackage[ruled,vlined]{algorithm2e}

\usepackage{textcomp}

\usepackage{kantlipsum}
\let\labelindent\relax
\usepackage[inline]{enumitem}
\usepackage[english]{babel}

\DeclareMathDelimiter{\orbrack}{\mathopen}{operators}{"5D}{largesymbols}{"03}
\DeclareMathDelimiter{\clbrack}{\mathclose}{operators}{"5B}{largesymbols}{"02}
\def\intcc#1{\ensuremath{[#1]}}
\def\intoo#1{\ensuremath{\orbrack#1\clbrack}}
\def\intoc#1{\ensuremath{\orbrack#1]}}
\def\intco#1{\ensuremath{[#1\clbrack}}

\def\Hintcc#1{\ensuremath{\llbracket#1\rrbracket}}

\def\defas{\ensuremath{\mathrel{:=}}}

\DeclareMathOperator{\id}{id}

\def\Set#1#2{\ensuremath{
\left\{#1\,\middle|\,#2\right\}
}}

\def\norm#1{\|  #1 \| }

\def\abs#1{|  #1 | }

\newtheorem{theorem}{Theorem}
\newtheorem{lemma}[theorem]{Lemma}

\newtheorem{remark}{Remark}
\newtheorem{problem}{Problem}
\newtheorem{assumption}{Assumption}

\newtheorem{definition}{Definition}

\begin{document}


\title{Underapproximating Safe Domains of Attraction for Discrete-Time Systems Using Implicit Representations of Backward Reachable Sets} 
\author{Mohamed Serry\thanks{M.  Serry and J. Liu are with the Department of Applied Mathematics, University of Waterloo, Waterloo, Ontario, Canada. (email:
\{mserry, j.liu\}@uwaterloo.ca).}  and 
 Jun Liu
 \thanks{This work was funded in part by the Natural Sciences and Engineering Research Council of Canada
and the Canada Research Chairs program.}}   
\date{}  
\IEEEoverridecommandlockouts

\maketitle
\thispagestyle{empty}

\begin{abstract}
Analyzing and certifying the   stability and attractivity of nonlinear systems is a topic of ongoing  research interest that has been extensively investigated by control theorists and engineers for many years. However, accurately estimating domains of attraction for nonlinear systems remains a challenging task, where existing  estimation methods tend to be conservative or limited to low-dimensional systems. In this work, we propose an iterative approach to accurately underapproximate  safe (state-constrained) domains of attraction for general discrete-time autonomous nonlinear systems. Our approach relies on  implicit representations of safe backward reachable sets of { initial} safe regions of attraction, where such initial regions can {be}  easily constructed using, e.g., quadratic Lyapunov functions. The iterations of our approach are monotonic (in the sense of set inclusion), {converging to the safe domain of attraction}. Each iteration results in a safe region of attraction, represented as a sublevel set, that underapproximates the safe domain of attraction. { The sublevel set representations of the resulting regions of attraction can be efficiently utilized in verifying  the inclusion of  given points of interest in the safe domain of attraction}. We illustrate our approach through two numerical examples, involving two- and four-dimensional nonlinear systems.
\end{abstract}
\section{Introduction}
When analyzing a control system, several crucial properties are typically sought to be ensured, including the  stability and attractivity of the system's equilibrium points. These properties provide  robustness guarantees when the system's equilibrium points  are  slightly perturbed due to external disturbances.  Besides, it is also important to ensure that the control system of interest  possesses invariance properties, with state values always lying within a specified safe  domain. 

Generally speaking,  nonlinear systems typically lack global stability and attractivity  properties, and invariance cannot be guaranteed for all the points of the safe domain (i.e., system trajectories starting from the safe domain may leave it). Hence, regions where such properties hold are estimated and considered when specifying  safe operational domains for dynamical systems. In this paper, we consider the problem of estimating the state-constrained or safe domain of attraction (DOA) of a general discrete-time autonomous nonlinear system. The safe DOA consists of the state values that are guaranteed to be driven to an equilibrium point of interest under the system's dynamics, while always satisfying specified safety state constraints.  

In the literature, DOAs are predominantly estimated using the framework of Lyapunov functions. This framework typically assumes candidate Lyapunov functions of fixed templates (e.g., quadratic forms and sum-of-squares polynomials). The parameters of such templates are then chosen to satisfy the standard Lyapunov conditions, or the  more relaxed multi-step  and  non-monotonic Lyapunov conditions \cite{ahmadi2008non,bobiti2014computation}.  This approach is generally restrictive  due to the use of fixed templates, providing, if existent, conservative estimates of  DOAs. 

DOAs can be  characterized  as sublevel sets of  particular Lyapunov functions that are unique solutions to some nonlinear functional equations (e.g., maximal Lyapunov, Zubov, and Bellman-type equations \cite{balint2006methods,giesl2007determination,oshea1964extension,xue2020characterization}). However, solutions to such functional equations are very difficult, if not impossible, to obtain analytically, and numerical solutions to such equations are limited to low-dimensional systems. In addition, numerical solutions to such equations do not necessarily provide certifiable DOA estimates (i.e., the resulting DOA estimates may not correspond to actual regions of attraction).

Recently, there has been a growing interest in using learning-based approaches to estimate DOAs, where neural networks are trained to satisfy standard Lyapunov conditions  and then verification tools (e.g., interval arithmetic and mixed-integer programming) are implemented to ensure that the trained neural networks provide certifiable DOA estimates \cite{wu2023neural,dai2021lyapunov,chen2021learning}.  Despite the high computational efficiency associated with training neural networks, neural network verification suffers from high computational demands due to state-space discretization. 

An important point to raise here is the following. For most of the approaches, where the estimates of  DOAs are given as sublevel sets of polynomial or neural network functions, pointwise inclusion can be verified efficiently, which basically requires function(s) evaluations at a point. However, using such set representations in verifying set inclusion (i.e, verifying  that a set of interest is contained in an estimate of the DOA) is computationally demanding, requiring the use of the badly-scaled verification tools mentioned above. The complexity of the sublevel set representations may also increase  the computational demands associated with the verification process.

In \cite{balint2006methods}, an interesting approach  was proposed to under-estimate  DOAs using backward reachable sets of carefully constructed  balls that are regions of attraction. The approach in \cite{balint2006methods} did not account for state constraints,  and it was designed particularly for continuously differentiable systems. In addition, the approach in $\cite{balint2006methods}$    did not provide a systemic way to represent backward reachable sets. 
 
 { In \cite{kothyari2024construction}, backward reachable set computations based on the method in \cite{yang2022efficient} were  utilized to obtain regions of null controllability. The approach in \cite{kothyari2024construction} did not account for state constraints and the  backward reachable sets were conservatively estimated through  linearization at the origin and approximations of the Pontryagin difference, hence, the method does not provide any convergence guarantees}. Recently, implicit  representations  of backward reachable sets have been utilized in computing invariant sets \cite{rakovic2022implicit}, and they have been shown to be efficient in verifying pointwise inclusion. However, and up to our knowledge, implicit representations   have not been adopted in approximations of DOAs.

Motivated by the utilities of  backward reachable sets and their implicit representations in set-based computations, we propose an iterative approach that provides arbitrarily precise underapproximations of the  safe (state-constrained) DOA of a general  discrete-time autonomous nonlinear system.  Each iteration of the proposed method results in a safe region of attraction,  with a sublevel set representation, that underapproximates the safe DOA. Such a set-level representation can be efficiently utilized in verifying pointwise inclusion.

The organization of this paper is as follows:  the necessary preliminaries and notation are introduced in Section \ref{sec:Preliminaries}, the relationship between  safe backward reachable sets  and safe DOAs is established and  a general iterative approach to compute safe backward reachable sets is discussed in Section \ref{sec:BRS}, the utilization of implicit representations of backward reachable sets in estimating  safe DOAs is illustrated in Section \ref{sec:ImplicitBRS}, 
a brief discussion on constructing  initial safe regions of attraction that can be used in the iterative approach is  introduced in Section \ref{sec:EllipsoidalROA}, the proposed method is illustrated  through two numerical examples in Section \ref{sec:NumericalExamples}, and the study is concluded in Section \ref{sec:Conclusion}.

\section{Notation and Preliminaries}
\label{sec:Preliminaries}
Let $\mathbb{R}$, $\mathbb{R}_+$,  $\mathbb{Z}$, and $\mathbb{Z}_{+}$ denote
the sets of real numbers, non-negative real numbers, integers, and
non-negative integers, respectively, and
$\mathbb{N} = \mathbb{Z}_{+} \setminus \{ 0 \}$.
Let $\intcc{a,b}$, $\intoo{a,b}$,
$\intco{a,b}$, and $\intoc{a,b}$
denote closed, open and half-open
intervals, respectively, with end points $a$ and $b$, and
 $\intcc{a;b}$, $\intoo{a;b}$,
$\intco{a;b}$, and $\intoc{a;b}$ stand for their discrete counterparts,
e.g.,~$\intcc{a;b} = \intcc{a,b} \cap \mathbb{Z}$, and
$\intco{1;4} = \{ 1,2,3 \}$.  In $\mathbb{R}^{n}
$, the relations $<$, $\leq$, $\geq$, and
$>$ are defined component-wise, e.g., $a < b$, where $a,b\in \mathbb{R}^{n}$, iff $a_i < b_i$ for
all $i\in  \intcc{1;n}$. For {$a, b \in \mathbb{R}^n$, $a \leq b$},
the closed hyper-interval (or hyper-rectangle) $\Hintcc{a,b}$ denotes the set $\Set{x\in \mathbb{R}^{n}}{a\leq x\leq b}$. Let $\norm{\cdot}$ and $\norm{\cdot}_{\infty}$ denote the Euclidean and maximal norms on $\mathbb{R}^{n}$, respectively, and $\mathbb{B}_{n}$ be the $n$-dimensional closed unit ball induced by $\norm{\cdot}$. The $n$-dimensional zero vector is denoted by $0_{n}$.  Let $\id_{n}$ denote the $n\times n$ identity matrix. For $A\in \mathbb{R}^{n\times m}$, $\norm{A}$ and  $\norm{A}_{\infty}$ denote the matrix norms of $A$ induced by the Euclidean and maximal norms, respectively. Given $x\in \mathbb{R}^{n}$  and $A\in \mathbb{R}^{n\times m}$, $\abs{x}\in \mathbb{R}^{n}_{+}$ and $\abs{A}\in \mathbb{R}_{+}^{n\times m}$ are defined as $|x|_{i}\defas |x_{i}|,~i\in \intcc{1;n}$, and $\abs{A}_{i,j}\defas \abs{A_{i,j}},~(i,j)\in \intcc{1;n}\times \intcc{1;m}$, respectively. Let $\mathcal{S}^{n}$ denote the set of $n\times n$ real symmetric matrices. Given $A\in \mathcal{S}^{n}$, $\underline{\lambda}(A)$ and $\overline{\lambda}(A)$ denote the minimum and maximum eigenvalues of $A$, respectively.  Let $\mathcal{S}_{++}^{n}$ denote the set of $n\times n$ real symmetric positive definite matrices. Given $A\in \mathcal{S}_{++}^{n}$, $A^{\frac{1}{2}}$ denotes the unique real symmetric positive definite matrix $K$ satisfying $A=K^2$ \cite[p.~220]{Abadir2005matrix}, and $A^{-\frac{1}{2}}\defas (A^{\frac{1}{2}})^{-1}$. Note that for $A\in \mathcal{S}^{n}_{++}$, $\underline{\lambda}(A)\norm{x}^{2}\leq x^{\intercal}Ax=\norm{A^{\frac{1}{2}}x}^{2}\leq \overline{\lambda}(A)\norm{x}^{2}$ for all $x\in \mathbb{R}^{n}$. The interior of $X\subseteq \mathbb{R}^{n}$, denoted by $\mathrm{int}(X)$,  is the set $\Set{x\in X}{\exists r\in \mathbb{R}_{+}\setminus\{0\}~\text{s.t.}~ x+r\mathbb{B}_{n}\subseteq X}$. Given $f\colon X\rightarrow Y$, $P\subseteq X$, and $Q\subseteq Y$, the image and preimage of $f$ on $P$ and $Q$ are defined as  $f(P)\defas\Set{f(x)}{x\in P}$ and $f^{-1}(Q)\defas\Set{x\in X}{f(x)\in Q}$, respectively.    Given $f\colon X \rightarrow X$ and  $x\in X$,  $f^{0}(x)\defas x$, and for  $M\in \mathbb{N}$, we define $f^{M}(x)$ recursively as follows:   $f^{k}(x)=f(f^{k-1}(x)),~k\in \intcc{1;M}$.  Given a  function $g\colon \mathbb{R}^{n}\rightarrow\mathbb{R}$ and $c\in \mathbb{R}$, $\Psi(g,c)$ denotes  the $c$-sublevel set of $g$ defined as 
    $\Psi(g,c)\defas \Set{x\in \mathbb{R}^{n}}{g(x)\leq c}
    $. 
\begin{lemma}\label{Lem:PreImageSLS}
    Given $g,h\colon \mathbb{R}^{n}\rightarrow\mathbb{R}$, $f\colon \mathbb{R}^{n}\rightarrow \mathbb{R}^{n}$, and $c\in \mathbb{R}$, we have 
   $
f^{-1}(\Psi(g,c))\cap \Psi(h,c)= \Psi(\tilde{g},c),
   $
   where
   $
\tilde{g}(\cdot)\defas \max\{h(\cdot),g(f(\cdot))\}.
   $
\end{lemma}
 {
 \begin{proof}
    $
        x\in f^{-1}(\Psi(g,c))\cap \Psi(h,c)\Leftrightarrow x\in \Psi(h,c)~ \&~ f(x)\in \Psi(g,c)\Leftrightarrow h(x)\leq c~ \&~g(f(x))\leq c \Leftrightarrow \tilde{g}(x)\leq c.
  $
   
 \end{proof}
}

\section{Problem Setup}
 Consider the discrete-time system:
\begin{equation}\label{eq:System}
x_{k+1}=f(x_{k}),~k\in \mathbb{Z}_{+},
\end{equation}
where $x_{k}\in \mathbb{R}^{n}$ is the state and $f\colon \mathbb{R}^{n}\rightarrow \mathbb{R}^{n}$ is the system's transition function. The trajectory of system \eqref{eq:System} starting from $x\in \mathbb{R}^{n}$ is the function $\varphi_{x}\colon\mathbb{Z}_{+}\rightarrow \mathbb{R}^{n}$, defined as follows:
 $   \varphi_{x}(0)=x$,
   $ \varphi_{x}(k+1)=f(\varphi_{x}(k))=f^{k+1}(x)$, $k\in \mathbb{Z}_{+}$.
Without loss of generality, we assume that:
\begin{assumption}\label{Assump:EqPoint}
$0_{n}$ is an equilibrium point of system \eqref{eq:System} (i.e., $f(0_{n})=0_{n}$ or $0_{n}$ is a fixed point of $f$).
\end{assumption}

Let $\mathcal{X}\subseteq\mathbb{R}^{n}$ be a fixed safe  set characterizing the state constraints to be imposed on system \eqref{eq:System}, where we assume that:
\begin{assumption}\label{Assump:0InInterior}
    $
    0_{n}\in \mathrm{int}(\mathcal{X}).
    $
\end{assumption}
{ Let $\mathcal{D}_{0}^{\mathcal{X}}\subseteq \mathbb{R}^{n}$ denote the safe DOA within $\mathcal{X}$, which is defined as:
$$
\mathcal{D}_{0}^{\mathcal{X}}\defas \left\{x\in \mathbb{R}^{n} \middle\vert 
     \varphi_{x}(k)\in \mathcal{X}~\forall k\in \mathbb{Z}_{+}, 
      \lim_{k\rightarrow \infty} \varphi_{x}(k)=0_{n}
\right\}.
$$
The set $\mathcal{D}_{0}^{\mathcal{X}}$ consists of the state values that are guaranteed to be driven to $0_{n}$  under the  dynamics of system \eqref{eq:System}, while remaining within $\mathcal{X}$ at all times. Note that, by  definition, $\mathcal{D}_{0}^{\mathcal{X}}\subseteq \mathcal{X}$.}
We assume that:
\begin{assumption}\label{Assump:0InTheInteriorofD}
    $0_{n}\in \mathrm{int}(\mathcal{D}_{0}^{\mathcal{X}})$.
\end{assumption}
\begin{definition}[Safe region of attraction] A subset $\mathcal{S}\subseteq \mathbb{R}^{n}$ is called a safe region of attraction (ROA)  within $\mathcal{X}$ iff 
     $\mathcal{S}\subseteq \mathcal{D}_{0}^{\mathcal{X}}$,
     $0_{n}\in \mathrm{int}(\mathcal{S})$, and 
     $\mathcal{S}$ is invariant under $f$ (i.e., $f(\mathcal{S})\subseteq \mathcal{S}$).
\end{definition}
{
Intuitively, a safe region of attraction $\mathcal{S}$ is a subset of $\mathcal{D}_{0}^{\mathcal{X}}$, containing $0_{n}$ in its interior, such that for any state $x$ in $\mathcal{S}$, the resulting trajectory stays within $\mathcal{S}$, while converging to the origin. 
}
\begin{problem} \label{Problem}
Let $\mathcal{V}\subseteq \mathcal{X}$ be a given  safe ROA  within $\mathcal{X}$\footnote{We show in Section \ref{sec:EllipsoidalROA}, how to compute a safe ROA for a case where $f$ is sufficiently smooth, and the origin is asymptotically stable.}. Our goal  in this paper is to compute a sequence of sets $\{\mathcal{X}_{k}\}_{k\in \mathbb{Z}}$ such that 
$\mathcal{X}_{k}$ is a safe ROA  within $\mathcal{X}$ for all $k\in \mathbb{Z}_{+}$,
$\mathcal{X}_{k}\subseteq \mathcal{X}_{k+1},~k\in \mathbb{Z}_{+},
$
and 
$
\bigcup_{k\in \mathbb{Z}_{+}}\mathcal{X}_{k}=\mathcal{D}_{0}^{\mathcal{X}}.
$
\end{problem}
Unless otherwise specified, system \eqref{eq:System}, the safe set $\mathcal{X}$, the safe ROA $\mathcal{V}$, and Assumptions \ref{Assump:EqPoint},  \ref{Assump:0InInterior}, and \ref{Assump:0InTheInteriorofD}   are fixed throughout the subsequent discussion. Further assumptions will be imposed in later sections.

\section{Safe Backward Reachable Sets and the Domain of Attraction}
\label{sec:BRS}
We start our attempt to address  Problem \ref{Problem} by introducing safe backward reachable sets. 
\begin{definition}[Safe backward reachable sets]
Given   $\mathcal{Y},\mathcal{Z}\subseteq \mathbb{R}^{n}$, the safe  \textit{any-time} backward reachable set of $\mathcal{Y}$ within $\mathcal{Z}$ is defined as:
\begin{equation}\label{eq:BRS}
\mathcal{R}^{\mathcal{Z}}_{-}(\mathcal{Y})\defas \left\{x\in \mathbb{R}^{n} \middle\vert \begin{array}{l}
    \exists N\in \mathbb{Z}_{+}~\mathrm{s.t.}~\varphi_{x}(k)\in \mathcal{\mathcal{Z}}\\
   \forall k\in \intcc{0;N},~\varphi_{x}(N)\in \mathcal{Y}
  \end{array}\right\}.
\end{equation}
\end{definition}

In the following theorem, we illustrate the iterative construction of safe any-time backward reachable sets using the preimage of the map $f$.
\begin{theorem}\label{thm:IterativeBRS}
   Let $\mathcal{Y},\mathcal{Z}\subseteq \mathbb{R}^{n}$, and let the sequence of sets $\{\mathcal{Y}_{k}\}_{k\in \mathbb{Z}_{+}}$ be defined as follows:  $\mathcal{Y}_{0}=\mathcal{Y}\cap \mathcal{Z}$, and 
       $\mathcal{Y}_{k+1}= f^{-1}(\mathcal{Y}_{k})\cap \mathcal{Z},~k\in \mathbb{Z}_{+}$.
   Then, 
   $
\mathcal{R}^{\mathcal{Z}}_{-}(\mathcal{Y})=\bigcup_{k\in \mathbb{Z}_{+}}\mathcal{Y}_{k}.
   $
\end{theorem}
\begin{proof}
   Let $x\in \mathcal{R}_{-}^{\mathcal{Z}}(\mathcal{Y})$. Then, there exists  $N\in \mathbb{Z}_{+}$ such that $\varphi_{x}(k)\in \mathcal{Z}$ for all  $k\in \intcc{0;N}$ and $\varphi_{x}(N)\in \mathcal{Y}$. We may assume without loss of generality that $N\geq 1$. We claim that $\varphi_{x}(N-k)\in \mathcal{Y}_{k}$ for all $k\in \intcc{0;N}$. As { $\varphi_{x}(N)\in \mathcal{Y}$ and $\varphi_{x}(N)\in \mathcal{Z}$}, we have {$\varphi_{x}(N)\in \mathcal{Y}\cap \mathcal{Z}=\mathcal{Y}_{0}$}, hence the claim holds for $k=0$. Assume that the claim holds for some $k\in \intcc{0;N-1}$, that is, $\varphi_{x}(N-k)\in \mathcal{Y}_{k}$. This implies that $f(\varphi_{x}(N-k-1))\in \mathcal{Y}_{k}$, and hence $\varphi_{x}(N-k-1)\in f^{-1}(\mathcal{Y}_{k})$. But $\varphi_{x}(N-k-1)\in \mathcal{Z}$, hence $\varphi_{x}(N-k-1)=\varphi_{x}(N-(k+1))\in f^{-1}(\mathcal{Y}_{k}) \cap \mathcal{Z}=\mathcal{Y}_{k+1}$, and that proves the inclusion claim. This subsequently indicates that $x=\varphi_{x}(0)\in \mathcal{Y}_{N}\subseteq \cup_{k\in \mathbb{Z}_{+}}\mathcal{Y}_{k}$. 

 {  Now, let $x \in \cup_{k\in \mathbb{Z}_{+}}\mathcal{Y}_{k}$. Then, $x\in \mathcal{Y}_{N}$ for some $N\in \mathbb{Z}_{+}$. If $N\in \{0,1\}$, it can easily be shown that $x\in \mathcal{R}^{\mathcal{Z}}_{-}(\mathcal{Y})$. Now, assume that $N\geq 2$. We claim that $\mathcal{\varphi}_{x}(k)\in \mathcal{Z}$ and $\varphi_{x}(k+1)\in \mathcal{Y}_{N-k-1}$ for all $k\in \intcc{0;N-1}$. By the definition of $\mathcal{Y}_{N}$ ($\mathcal{Y}_{N}=f^{-1}(\mathcal{Y}_{N-1})\cap \mathcal{Z}$), we have $\varphi_{x}(0)=x\in \mathcal{Z}$ and $\varphi_{x}(1)=f(x)\in \mathcal{Y}_{N-1}$, hence the claim holds for $k=0$. Assume that $\mathcal{\varphi}_{x}(k)\in \mathcal{Z}$ and $\varphi_{x}(k+1)\in \mathcal{Y}_{N-k-1}$ hold for some $k\in \intcc{0;N-2}$. By the definition of $\mathcal{Y}_{N-k-1}$ ($\mathcal{Y}_{N-k-1}=f^{-1}(\mathcal{Y}_{N-k-2})\cap \mathcal{Z}$), we have $\varphi_{x}(k+1)\in \mathcal{Z}$ and $f(\varphi_{x}(k+1))=\varphi_{x}(k+2)\in \mathcal{Y}_{N-k-2}$, and the claim holds by induction for all $k\in \intcc{0;N-1}$. Therefore, we have   $\varphi_{x}(k)\in \mathcal{Z}$ for all $k\in \intcc{0;N-1}$ and $\varphi_{x}(N)\in \mathcal{Y}_{0}$. As $\mathcal{Y}_{0}=\mathcal{Y}\cap \mathcal{Z}$,  $\varphi_{x}(N)\in \mathcal{Z}$ and $\varphi_{x}(N)\in \mathcal{Y}$. Hence, $x\in \mathcal{R}_{-}^{\mathcal{Z}}(\mathcal{Y})$ and that completes the proof.}  
\end{proof}

In the next result, we show how the state-constrained preimage of a safe ROA preserves its invariance and safe attractivity to the origin.
\begin{theorem} \label{thm:OneStepBRSofROA}
     Let $\mathcal{S}\subseteq \mathbb{R}^{n}$ be a safe ROA  within $\mathcal{X}$ and $\mathcal{T}=f^{-1}(\mathcal{S})\cap \mathcal{X}$.  Then, 
  $
\mathcal{S} \subseteq  \mathcal{T}
  $
  and 
  $\mathcal{T}$ is also a safe ROA  within $\mathcal{X}$.
\end{theorem}
\begin{proof}
 Let $x\in \mathcal{S}$. The invariance of  $\mathcal{S}$  under $f$ implies that $f(x)\in \mathcal{S}$ and hence, $x\in f^{-1}(\mathcal{S})$. As $\mathcal{S}\subseteq \mathcal{X}$, we then have $x\in f^{-1}(\mathcal{S})\cap \mathcal{X}=\mathcal{T}$, and that proves the first claim. Consequently, we have $0_{n}\in \mathrm{int}(\mathcal{T})$. Note that for any $x\in \mathcal{T}$, $f(x)\in \mathcal{S}\subseteq \mathcal{T}$, which implies the invariance of $\mathcal{T}$. Finally, let $x\in \mathcal{T}$.  By the invariance of $\mathcal{T}$, we have $\varphi_{x}(k)\in \mathcal{T}\subseteq\mathcal{X}~\forall k\in \mathbb{Z}_{+}$, and, by the  definition of $\mathcal{T}$,  $\varphi_{x}(1)=f(x)\in \mathcal{S}$, implying  $\lim_{k\rightarrow \infty}\varphi_{f(x)}(k)=\lim_{k\rightarrow \infty}\varphi_{x}(k+1)=0_{n}$. Hence,  $x\in \mathcal{D}_{0}^{\mathcal{X}}$. 
\end{proof}

Now, we establish the relationship between safe backward reachable sets and the safe DOA $\mathcal{D}_{0}^{\mathcal{X}}$. 
\begin{theorem}\label{thm:BRS=DOA}
  Let $\mathcal{S}\subseteq \mathbb{R}^{n}$ be a safe ROA  within $\mathcal{X}$.  Then, 
  $
\mathcal{R}^{\mathcal{X}}_{-}(\mathcal{S})=\mathcal{D}_{0}^{\mathcal{X}}.
  $
\end{theorem}
\begin{proof}
    Let $x\in \mathcal{D}_{0}^{\mathcal{X}}$.  Using the   definition of $\mathcal{D}_{0}^{\mathcal{X}}$ and   the fact that  $0_{n}$ is in the interior of {$\mathcal{S}$},  there exists $N\in \mathbb{Z}_{+}$, such that $\varphi_{x}(k)\in \mathcal{X}$ for all $k\in \intcc{0;N}$ and {$\varphi_{x}(N)\in \mathcal{S}$}, hence { $x\in \mathcal{R}^{\mathcal{X}}_{-}(\mathcal{S})$}. On the other hand, for ${x\in \mathcal{R}^{\mathcal{X}}_{-}(\mathcal{S})}$,  there exists $N\in \mathbb{Z}_{+}$, such that $\varphi_{x}(N)\in \mathcal{S}$ and $\varphi_{x}(k)\in \mathcal{X}$ for all $k\in \intcc{0;N}$. As $\mathcal{S}$ is a safe ROA, $\varphi_{x}(k)\in \mathcal{S}\subseteq \mathcal{X}$ for all $k\in \intco{N+1;\infty}$ with $\lim_{k\rightarrow \infty}\varphi_{x}(k+N)=0_{n}$. Hence,   $x\in \mathcal{D}_{0}^{\mathcal{X}}$.
\end{proof}
In view of Theorems \ref{thm:IterativeBRS}, \ref{thm:OneStepBRSofROA}, and \ref{thm:BRS=DOA}, and by using an inductive argument, we have the following result, which elucidates how the iterative computations of safe backward reachable sets enable arbitrarily precise underapproximations of the safe DOA. 
\begin{theorem}\label{thm:IterativeComputationsROA}
Define the sequence $\{\mathcal{V}_{k}\}_{k\in \mathbb{Z}_{+}}$ as follows: $\mathcal{V}_{0}=\mathcal{V}$ and $\mathcal{V}_{k+1}=f^{-1}(\mathcal{V}_{k})\cap \mathcal{X},~k\in \mathbb{Z}_{+}$. Then, each $\mathcal{V}_{k}$ is a safe ROA  within $\mathcal{X}$ for all $k\in \mathbb{Z}_{+}$, $\mathcal{V}_{k}\subseteq\mathcal{V}_{k+1}$ for all $k\in \mathbb{Z}_{+}$, and $\bigcup_{k\in \mathbb{Z}_{+}}\mathcal{V}_{k}=\mathcal{D}_{0}^{\mathcal{X}}$.
\end{theorem}

\begin{remark}[Subsets of safe ROAs are useful]
 \label{rem:UASafeROA}   
Theorem \ref{thm:IterativeComputationsROA} provides an iterative approach that yields safe ROAs. While subsets of safe ROAs may not possess invariance properties, they ensure safety and attractivity . This is due to the fact, which follows from the definition of safe ROAs, that  for a safe ROA $\mathcal{S}$  and a subset $\mathcal{T}\subseteq \mathcal{S}$,  $\varphi_{x}(k)\in \mathcal{X}~\forall k\in \mathbb{Z}_{+}$, and  
      $\lim_{k\rightarrow \infty} \varphi_{x}(k)=0_{n}$ for all   $x\in \mathcal{T}$.
This  indicates that if the sets $\mathcal{V}_{k},k\in \mathbb{Z}_{+}$, in Theorem \ref{thm:IterativeComputationsROA} cannot be computed exactly, they can be replaced by underapproximations, which still provide safe attraction guarantees. 
\end{remark}
\section{Backward Reachable Sets: Implicit Representations}
\label{sec:ImplicitBRS}
In the previous section, we highlighted the general framework to underapproximate  $\mathcal{D}_{0}^{\mathcal{X}}$. Herein, we {provide} the sublevel set representations of the resulting underapproximations, where we impose the following additional assumption: 
\begin{assumption}\label{Assump:XandVsublevelsets}
   The sets $\mathcal{X}$ and $\mathcal{V}$ are 1-sublevel sets of the given  functions $\theta\colon \mathbb{R}^{n}\rightarrow \mathbb{R}$ and $v\colon \mathbb{R}^{n}\rightarrow \mathbb{R}$, respectively (i.e., $\mathcal{X}=\Psi(\theta,1),~\mathcal{V}=\Psi(v,1)$).
\end{assumption}

The next result provides closed-form formulas for the sublevel set representations of the underapproximations obtained in Theorem \ref{thm:IterativeComputationsROA}.

\begin{theorem}\label{thm:ImplicitRepresentation}  Define the sequence $\{\mathcal{V}_{k}\}_{k\in \mathbb{Z}_{+}}$ as in Theorem \ref{thm:IterativeComputationsROA}. Then,
    $
\mathcal{V}_{k}=\Set{x\in \mathbb{R}^{n}}{v_{k}(x)\leq 1},~k\in \mathbb{Z}_{+},
    $
    where the functions $v_{k}\colon \mathbb{R}^{n}\rightarrow \mathbb{R}$, $k\in \mathbb{Z}_{+}$, are defined as follows: $v_{0}(\cdot)=v(\cdot)$, and, for all $x\in \mathbb{R}^{n}$ and $~k\in \mathbb{Z}_{+}$,
    \begin{equation}\label{eq:VkFormula1}
v_{k+1}(x)=\max\{\theta(x),v_{k}(f(x))\}
\end{equation}
or, explicitly,
\begin{equation}\label{eq:VkFormula2}
v_{k+1}(x)=\max\{\max_{i\in \intcc{0;k}} \theta(\varphi_{x}(i)), v(\varphi_{x}(k+1))\}.
    \end{equation}
\end{theorem}
\begin{proof}
 The definition of $v_{0}$ follows trivially from the fact that $\mathcal{V}_{0}=\mathcal{V}$ and  Assumption \ref{Assump:XandVsublevelsets}.  Equation \eqref{eq:VkFormula1} follows from the definition of $\{\mathcal{V}_{k}\}_{k\in \mathbb{Z}_{+}}$ and Lemma \ref{Lem:PreImageSLS}. To prove \eqref{eq:VkFormula2}, we use induction. Let $x\in \mathbb{R}^{n}$. For $k=0$ (i.e., for  $v_{1}$), we have, using \eqref{eq:VkFormula1}, 
$
v_{1}(x)=\max\{\theta(x),v_{0}(f(x))\}=\max\{\theta(\varphi_{x}(0)),v_{0}(\varphi_{x}(1))\}
=\max\{\max_{j\in \intcc{0;0}}\theta(\varphi_{x}(j)),v(\varphi_{x}(1))\}$. Assuming \eqref{eq:VkFormula2} holds for some $k\in \mathbb{Z}_{+}$, then it holds for $k+1$ (i.e., for $v_{k+2}$) as follows: by   \eqref{eq:VkFormula1}, we have $ v_{k+2}(x)=\max\{\theta(x),v_{k+1}(f(x))\}$, where
$
    v_{k+1}(f(x))=\max\{\max_{i\in \intcc{0;k}} \theta(\varphi_{f(x)}(i)), v(\varphi_{f(x)}(k+1))\}= \max\{\max_{i\in \intcc{1;k+1}} \theta(\varphi_{x}(i)), v(\varphi_{x}(k+2))\}.
$
It then follows that
$
   v_{k+2}(x)= \max\{\max_{i\in \intcc{0;k+1}} \theta(\varphi_{x}(i)), v(\varphi_{x}(k+2))\},
$
and that completes the proof.
\end{proof}
\subsection{Efficient pointwise evaluation}
Theorem \ref{thm:ImplicitRepresentation} and, in particular, equation  \eqref{eq:VkFormula1} provide a pathway for efficient pointwise evaluations of the functions $v_{k},~k\in \mathbb{N}$, characterizing the underapproximations of the safe DOA $\mathcal{D}_{0}^{\mathcal{X}}$. Such  evaluations can be done recursively as illustrated in Algorithm \ref{Alg:EvaluatingVk}, which is adapted from \cite{rakovic2022implicit}.   
\begin{algorithm}
\SetAlgoLined
\KwData{$x\in \mathbb{R}^{n},~k\in \mathbb{N},~\theta,~v,~f$}
  $y \gets x$, $i\gets 1$, $Z\gets 0_{k+1}$

 \While{$i\leq k$}{
 
$ Z_{i}\gets \theta(y)$,
$y \gets f(y)$, 
$i\gets i+1$}
$Z_{i}\gets v(y)$,
 $v_{k}(x)\gets \max_{i\in \intcc{1;k+1}} Z_{i} $ 
 
\KwResult{$v_{k}(x)$}
 \caption{Evaluating $v_{k}(x)$}
 \label{Alg:EvaluatingVk}
 \end{algorithm}

\begin{remark}[Underapproximating sublevel sets]
We  observe from Theorem \ref{thm:ImplicitRepresentation} that the complexity of the formulas of $v_{k},~k\in \mathbb{Z}_{+}$, increases as $k$ increases. However, this increase in complexity does not have significant detrimental effect when it comes to pointwise evaluations, which can be done recursively according to Algorithm \ref{Alg:EvaluatingVk}. If it is of interest to impose bounded complexity (e.g., to enable a relatively scalable set inclusion verification), the functions $v_{k},~k\in \mathbb{Z}_{+}$,  can be replaced with bounding functions $\tilde{v}_{k},~k\in \mathbb{Z}_{+}$ (i.e., $\tilde{v}_{k}(\cdot)\geq {v}_{k}(\cdot)$), with reduced complexity. The  1-sublevel sets of the bounding functions are subsets of the safe ROAs $\mathcal{V}_{k},~k\in \mathbb{Z}_{+}$\footnote{If $\tilde{v}_{k}(\cdot)\geq v_{k}(\cdot)$, then $\tilde{v}_{k}(x)\leq 1$, for $x\in \mathbb{R}^{n}$, implies that  ${v}_{k}(x)\leq 1$. Hence, $\{x\in \mathbb{R}^{n}|\tilde{v}_{k}(x)\leq 1\}\subseteq \{x\in \mathbb{R}^{n}|v_{k}(x)\leq 1\}=\mathcal{V}_{k}$. }, hence they provide safe attractivity guarantees as highlighted in Remark \ref{rem:UASafeROA}. 

Fix $k \in \mathbb{N}$, { where the value of $k$ should not be intolerably large to enable handling the function compositions $f^{i}(\cdot),~i\in \intcc{0;k}$, symbolically}. We may find a function $\tilde{v}_{k}$, with a fixed template (e.g., sum-of-squares polynomial of specified degree), that bounds $v_{k}$ as follows. Using equation \eqref{eq:VkFormula2}, $\tilde{v}_{k}$ should satisfy 
$
\tilde{v}_{k}(\cdot) \geq \theta(f^{i}(\cdot))~\forall i\in \intcc{0,k-1}, 
$
and 
$
\tilde{v}_{k}(\cdot) \geq v(f^{k}(\cdot)).
$
{ In the case when $f$, $\theta$, and $v$ are polynomial functions, these inequalities can be cast  as constraints of a sum-of-squares  optimization problem that results in $\tilde{v}_{k}$.}

\end{remark}
\begin{remark}[Controlled systems]
    While our approach is restricted to  discrete-time autonomous systems, it can be useful in providing estimates of  safe null controllability domains for  discrete-time controlled systems. This can be done by integrating  a controlled system with a stabilizing feedback controller and then analyzing  the safe DOA of the closed-loop system. In Section \ref{sec:NumericalExamples}, we provide an example illustrating this idea.
\end{remark}
\section{Initial Safe Region of Attraction}
\label{sec:EllipsoidalROA}
In this section, we demonstrate how to obtain a safe ROA $\mathcal{V}$ using quadratic Lyapunov functions under the following additional assumptions on $f$: 
\begin{assumption} \label{Assump:fSmooth}
    $f$ is twice continuously differentiable over $\mathbb{R}^{n}$, and all the eigenvalues of the Jacobian of $f$ at $0_{n}$, denoted by $Df(0)$, are located  in the open unit ball 
 of the complex plane.  
\end{assumption}
\begin{remark}
    It should be noted that our construction of $\mathcal{V}$ herein can be replaced by  other constructions from the literature that are suited for systems that do not satisfy Assumption \ref{Assump:fSmooth}, as our framework highlighted by Theorems \ref{thm:IterativeComputationsROA} and \ref{thm:ImplicitRepresentation} is quite general.
\end{remark}

Our discussion herein adapts the Lyapunov analysis  in  \cite{bof2018lyapunov}.
Let $A=Df(0)$ and rewrite $f$ as 
$f(x)=Ax+h(x),~x\in \mathbb{R}^{n},$
where $h(\cdot)=f(\cdot)-A(\cdot)$. Let $Q\in \mathcal{S}^{n}_{++}$ be given, and $P\in \mathcal{S}^{n}_{++}$ be  the solution to the discrete-time algebraic Lyapunov equation
 $
A^{\intercal}PA-P=-Q.
 $
 Define the candidate Lyapunov function ${\nu}\colon \mathbb{R}^{n}\rightarrow \mathbb{R}$ as 
 $
{\nu}(x)\defas x^{\intercal}Px
=\norm{P^{\frac{1}{2}}x}^{2},~x\in \mathbb{R}^{n}$, which is positive definite over $\mathbb{R}^{n}$. Then, for all $~x\in \mathbb{R}^{n}$, 
$
\nu(f(x))-\nu(x)=-x^{\intercal}Qx+2x^{\intercal}A^{\intercal}P h(x)+h^{\intercal}(x)Ph(x).
$
Let $\mathcal{B}\subseteq \mathcal{X}$ be a hyper-rectangle with vector radius $R_{\mathcal{B}}\in \mathbb{R}^{n}_{+}\setminus \{0_{n}\}$, i.e., $\mathcal{B}=\Hintcc{-R_{\mathcal{B}},R_{\mathcal{B}}}$ (such a hyper-rectangle exists due to Assumption \ref{Assump:0InInterior}).  We can find a vector $\eta_{\mathcal{B}}\in \mathbb{R}_{+}^{n}$ (by bounding the Hessian of $f$ over $\mathcal{B}$, e.g, using  interval arithmetic) such that 
$
\abs{h(x)}\leq \frac{\norm{x}^{2}}{2}\eta_{\mathcal{B}},~x\in \mathcal{B}.
$
The above bound can be used in providing estimates of  $2x^{\intercal}A^{\intercal}P h(x)$ and $h^{\intercal}(x)Ph(x)$ as follows:  
$
2x^{\intercal}A^{\intercal}P h(x)
\leq \sqrt{\nu(x)}\norm{\abs{P^{\frac{1}{2}}}\eta_{\mathcal{B}}}\norm{P^{\frac{1}{2}}AP^{-\frac{1}{2}}}\norm{x}^{2}
$
and 
$
h^{\intercal}(x)Ph(x) 
\leq \norm{P}\norm{\eta_{\mathcal{B}}}^{2}\frac{\norm{x}^{2}\nu(x)}{4\underline{\lambda}(P)},
$
$x\in \mathcal{B}$. By defining $d=\underline{\lambda}(Q)-\varepsilon$, where $\varepsilon>0$ is a small parameter, we consequently have, for $x\in \mathcal{B}$,
$
\nu(f(x))-\nu(x)\leq -\varepsilon \norm{x}^{2}+\sqrt{\nu(x)}\norm{\abs{P^{\frac{1}{2}}}\eta_{\mathcal{B}}}\norm{P^{\frac{1}{2}}AP^{-\frac{1}{2}}}\norm{x}^{2}
+\norm{P}\norm{\eta_{\mathcal{B}}}^{2}\frac{\norm{x}^{2}\nu(x)}{4\underline{\lambda}(P)}
-d \norm{x}^{2}$.

We need to search for our safe ROA  $\mathcal{V}$ within $\mathcal{B}$, where we represent $\mathcal{V}$ as a sublevel set of $\nu$ (i.e., $\mathcal{V}=\Psi(\nu,c)$ for some $c\in \mathbb{R}_{+}\setminus \{0\}$), and we ensure the Lyapunov condition $\nu(f(x))-\nu(x)<0$ holds for all $x\in\mathcal{V}\setminus \{0_{n}\}$. To find $\mathcal{V}$ (or equivalently $c$), we impose that 
$
\sqrt{\nu(x)}\norm{\abs{P^{\frac{1}{2}}}\eta_{\mathcal{B}}}\norm{P^{\frac{1}{2}}AP^{-\frac{1}{2}}}\norm{x}^{2} +\norm{P}\norm{\eta_{\mathcal{B}}}^{2}\frac{\norm{x}^{2}\nu(x)}{4\underline{\lambda}(P)}-d \norm{x}^{2} \leq 0
$
or 
$
\sqrt{\nu(x)}\norm{\abs{P^{\frac{1}{2}}}\eta_{\mathcal{B}}}\norm{P^{\frac{1}{2}}AP^{-\frac{1}{2}}}+\norm{P}\norm{\eta_{\mathcal{B}}}^{2}\frac{\nu(x)}{4\underline{\lambda}(P)}\leq d 
$
for all $x\in \mathcal{V}$. This is fulfilled if $c \leq c_{1}$, where
$
c_{1}\defas (-\beta+\sqrt{\beta^2+4\alpha d})^2/(2\alpha)^{2},
$ $\alpha\defas\frac{\norm{P}\norm{\eta_{\mathcal{B}}}^{2}}{4\underline{\lambda}(P)}$ and $\beta\defas \norm{|P^{\frac{1}{2}}|\eta_{\mathcal{B}}}\norm{P^{\frac{1}{2}}AP^{-\frac{1}{2}}}$.  
To ensure that $\mathcal{V}$ is inside $\mathcal{B}$, we  impose the condition 
$
\nu(x)\leq { \mathcal{R}_{\mathcal{B},i}^{2}}/{P^{-1}_{i,i}},~i\in \intcc{1;n},
$
for all $x\in \mathcal{V}$, and that is guaranteed if $c\leq c_{2}$, where  
$
c_{2}\defas \min_{i\in \intcc{1;n}} { \mathcal{R}_{\mathcal{B},i}^{2}}/{P^{-1}_{i,i}}.
$
We can then choose 
$
c=\min\{c_{1},c_{2}\}$. 
To write $\mathcal{V}$ as a 1-sublevel set, we  define $v\colon \mathbb{R}^{n}\rightarrow \mathbb{R}$ as 
$
v(x)=\nu(x)/c,~x\in \mathbb{R}^{n}.
$
 Then, our safe ROA $\mathcal{V}$ is given by
 $
\mathcal{V}=\{x\in \mathbb{R}^{n}| v(x)\leq 1\}\subseteq \mathcal{X}.
 $

\section{Numerical Examples}
\label{sec:NumericalExamples}
In this section, we illustrate our approach through two numerical examples. Our  proposed method is implemented in MATLAB. In our computations of the initial safe ROAs according to Section \ref{sec:EllipsoidalROA}, the vector $\eta_{\mathcal{B}}$ is obtained using interval arithmetic bounds using the reachability software CORA  \cite{althoff2015introduction}, and the matrix $P$ is obtained using the MATLAB function \texttt{dlyap}.

\subsection{Two-machine system}
We consider a discrete version of the two-dimensional two-machines power system studied in \cite{vannelli1985maximal,willems1968improved}. The discrete version is obtained through Euler discretization and is given by \eqref{eq:System}, with
$
f(x_{k})=    \begin{pmatrix}
    x_{1,k}+\Delta_{t}  x_{2,k}\\
       x_{2,k}-\Delta_{t}( \frac{x_{2,k}}{2}+\sin(x_{1,k}+\frac{\pi}{3})-\sin(\frac{\pi}{3}))
    \end{pmatrix},
$
where the time step $\Delta_{t}$ is set to be $0.1$. We aim to estimate the safe DOA  $\mathcal{D}_{0}^{\mathcal{X}}$, where $\mathcal{X}=\Hintcc{-[1~0.5]^{\intercal},[1~0.5]^{\intercal}}$. Note that $\mathcal{X}$ can be written as a 1-sublevel set, with $\theta$ given by $\theta(x)=\norm{E x}_{\infty},~x\in \mathbb{R}^{2}$, where 
$
E=\mathrm{diag}([1~2]^{\intercal}).
$
Following the procedure described in Section \ref{sec:EllipsoidalROA}, we set $\mathcal{B}=\mathcal{X}$,  $Q=\id_{2}$, and we  obtained 
$
P =
   \begin{pmatrix} 
   21.9377 &  10.8408\\
   10.8408  & 33.6321
\end{pmatrix},
$
 $c=2.9345$, and a safe ROA  $\mathcal{V}=\Set{x\in \mathbb{R}^{2}}{x^{\intercal}Px/c\leq 1}$. We then computed  safe ROAs according to Theorem \ref{thm:IterativeComputationsROA} with $80$ iterations, where their sublevel set representations are given by Theorem \ref{thm:ImplicitRepresentation}. The safe ROAs $\mathcal{V}_{0},~\mathcal{V}_{30},~\mathcal{V}_{60}$, and $\mathcal{V}_{80}$ are depicted in Fig.  \ref{fig:two_machine_system}. Observe the monotonicity of the computed  ROAs and their satisfaction of the state constraints given by the set $\mathcal{X}$. We picked three initial conditions
 $
x_{0}^{(1)}=[1~-0.2]^{\intercal}
 $, $x_{0}^{(2)}=[-0.2~0.5]^{\intercal}$, and  $x_{0}^{(3)}=[-1~0]^{\intercal}$ inside the safe set $\mathcal{X}$, and we verified, using Algorithm \ref{Alg:EvaluatingVk}, that $x_{0}^{(1)},~x_{0}^{(2)}\in \mathcal{V}_{80}$, but $x_{0}^{(3)}\notin  \mathcal{V}_{80}$. Then, we generated trajectories starting from the picked initial conditions. Figure \ref{fig:two_machine_system} shows how the trajectories starting from  $x_{0}^{(1)}$ and  $x_{0}^{(2)}$ stay in the safe set $\mathcal{X}$ and converge to $0_{2}$, whereas the trajectory starting from $x_{0}^{(3)}$ leaves the safe set before returning to it and then converging to $0_{2}$. This highlights the usefulness of the safe ROAs obtained by our proposed approach in providing safe attraction guarantees.
\begin{figure}
    \centering
    \includegraphics[width=0.8\linewidth]{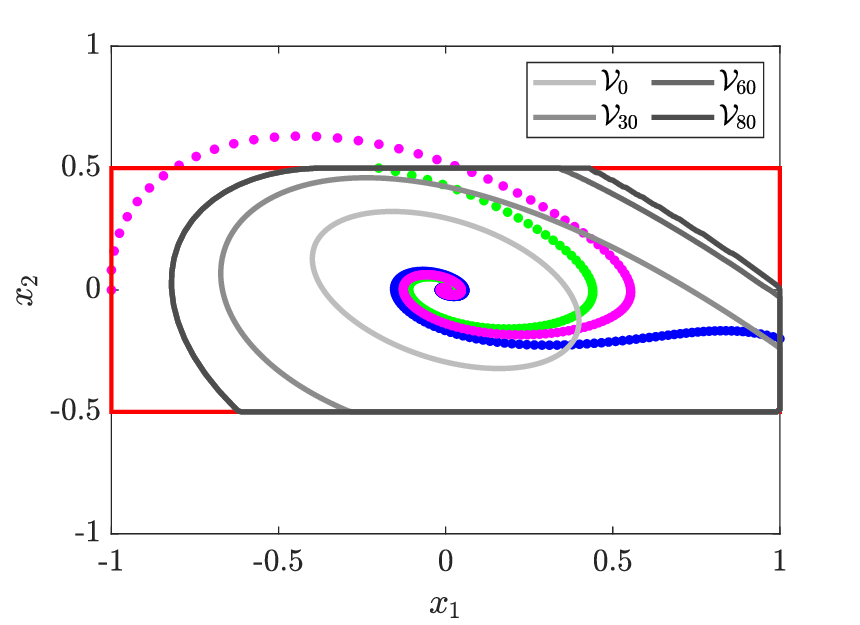}
    \caption{Estimates of the safe DOA of the two-machine system inside $\mathcal{X}$ (red), and generated trajectories starting from $\mathcal{X}$ (blue, green, and magenta). }
    \label{fig:two_machine_system}
\end{figure}
\subsection{Cart-pole system}
Herein, we consider a discrete-time version of the four-dimensional controlled cart-pole system given in \cite{spong1996energy} of the form $x_{k+1}=\tilde{f}(x_{k},u_{k})$, where 
\begin{equation}\label{eq:CartPole}
     \tilde{f}(x_{k},u_{k}) =    \begin{pmatrix}
    x_{1,k}+\Delta_{t}  x_{2,k}\\
x_{2,k}+\Delta_{t}u_{k}\\
       x_{3,k}+\Delta_{t} x_{4,k}\\
       x_{4,k}+\Delta_{t}(\sin(x_{3,k})-\cos(x_{3,k})u_{k})
    \end{pmatrix},    
\end{equation}
 $x_{1,k}$ and $x_{2,k}$ are the cart's position and velocity, respectively, $x_{3,k}$ and $x_{4,k}$, are the pole's angle (measured from the upright position) and angular velocity, respectively, and $u_{k}$ is the control input. The time step $\Delta_{t}$ is set to be $0.1$. It is required that the states of the system stay in the safe set $
\tilde{\mathcal{X}}=\Hintcc{-R_{\tilde{\mathcal{X}}},R_{\tilde{\mathcal{X}}}},
$
where
$
R_{\tilde{\mathcal{X}}}=\begin{bmatrix}
    0.1&0.1&\frac{\pi}{4}&0.1
\end{bmatrix}^{\intercal},
$
and  the control input satisfies the constraint
$
u_{k}\in \mathcal{U}=\intcc{-1,1}~\forall k\in \mathbb{Z}_{+}.
$
Our goal herein is to safely stabilize the system around the origin by implementing a linear feedback control and estimate the DOA of the closed-loop system, where the state and input constraints are fulfilled. We linearized the system at the origin, with zero control input, and computed a state-feedback control, through solving a  discrete-time algebraic Riccati equation\footnote{We used the MATLAB function \texttt{idare} to solve the mentioned equation.}, which resulted in the gain matrix 
$
K =
 \begin{bmatrix}  1.6897&   6.2464&  11.3886&  11.4026
\end{bmatrix}.
$
We then substituted $u_{k}=K x_{k}$ into \eqref{eq:CartPole}, and we obtained a closed-loop system of the form \eqref{eq:System}, where $f(x_{k})=\tilde{f}(x_{k},K x_{k})$. For the closed-loop system, the safe set $\mathcal{X}$ accounts for the state and input constraints of the open-loop system and is given as a 1-sublevel set of the function $\theta$ given by 
$
\theta(x)=\norm{\begin{pmatrix}
    K^{\intercal} &
    E^{\intercal}
\end{pmatrix}^{\intercal}x}_{\infty},~x\in \mathbb{R}^{4},
$
where $E=(\mathrm{diag}(R_{\tilde{\mathcal{X}}}))^{-1}$. We then followed the procedure given in Section \ref{sec:EllipsoidalROA} to obtain an ellipsoidal safe ROA. A hyper-rectangle that can be used in the estimation of the initial safe ROA $\mathcal{V}$ is given by $
\mathcal{B}=\left\{x\in \mathbb{R}^{4}\vert \norm{\begin{pmatrix}
    K^{\intercal}&
    E^{\intercal}
\end{pmatrix}^{\intercal}}_{\infty} \norm{x}_{\infty}\leq 1\right\}$.
We set $Q=\id_{4}$, and we obtained
$$
P=   \begin{pmatrix}
35.6188&   56.5630&   60.0805&   59.9877\\
   56.5630&  135.0700&  147.0047&  146.5897\\
   60.0805&  147.0047&  174.9002&  163.9973\\
   59.9877&  146.5897&  163.9973&  163.6202
   \end{pmatrix},
$$
 $c=0.0312$, and an initial safe ROA  $\mathcal{V}=\Set{x\in \mathbb{R}^{4}}{x^{\intercal}Px/c\leq 1}$. We  computed  safe ROAs according to Theorem \ref{thm:IterativeComputationsROA} with $60$ iterations, where their sublevel set representations are given by Theorem \ref{thm:ImplicitRepresentation}.
 \begin{figure}
    \centering
    \includegraphics[width=0.49\linewidth]{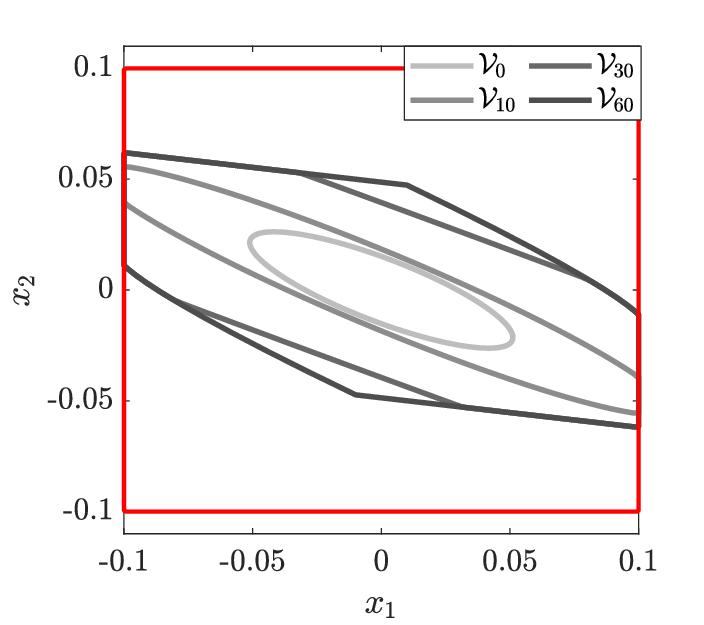} \includegraphics[width=0.49\linewidth]{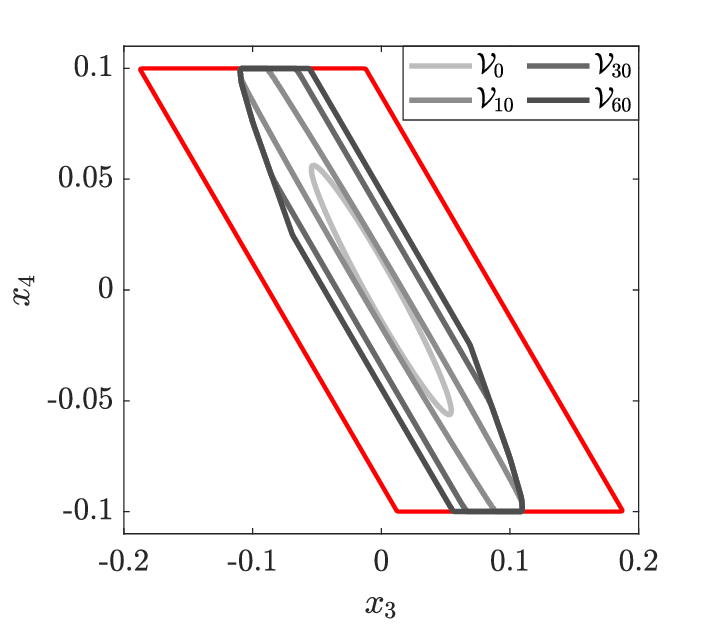}
    \caption{$x_{1}-x_{2}$ (left) and $x_{3}-x_{4}$ (right) cross-sections, with $(x_{3},x_{4})=(0,0)$ and $(x_{1},x_{2})=(0,0)$, respectively, of  the safe ROAs of the closed-loop cart-pole system inside $\mathcal{X}$ (red).}
    \label{fig:cart_pole}
\end{figure}
 Cross-sections of the  safe ROAs $\mathcal{V}_{0},~\mathcal{V}_{10},~\mathcal{V}_{30}$, and $\mathcal{V}_{60}$ are depicted in Fig. \ref{fig:cart_pole}. Fig. \ref{fig:cart_pole} displays the monotonicity of the resulting safe ROAs, with respect to set inclusion, and the satisfaction of the state constraints of the closed-loop cart-pole system.  We picked two initial conditions
 $
x_{0}^{(1)}=[0.1~-0.02~0~0]^{\intercal}
 $ and  $x_{0}^{(2)}=[-0.05~-0.05~0~0]^{\intercal}$ inside the set $\mathcal{X}$, and we verified, using Algorithm \ref{Alg:EvaluatingVk}, that $x_{0}^{(1)}\in \mathcal{V}_{60}$, but $x_{0}^{(2)}\notin  \mathcal{V}_{60}$. Then, we generated trajectories starting from the picked initial conditions, where the safety and convergence of the generated trajectories are verified by evaluating the function $\theta_{0}$ along the generated trajectories. Fig. \ref{fig:Theta_profile} shows how the trajectory starting from $x_{0}^{(1)}$ stays inside the safe set $\mathcal{X}$, converging to the origin \footnote{The safe set  $\mathcal{X}$ for the closed-loop system is compact, with the origin being in its interior, and the associated function $\theta$  is continuous, satisfying $\theta(0_{4})=0$ and $\theta(x)>0~\forall x\in \mathbb{R}^{4}\setminus \{0_{4}\}$. This implies that if $\{y_{k}\}_{k\in \mathbb{Z}_{+}}$ is a sequence with values in $\mathcal{X}$, and $\lim_{k\rightarrow \infty}\theta(y_{k})= 0$, then $\lim_{k\rightarrow \infty}y_{k}= 0_{4}$.}, whereas the trajectory starting from $x_{0}^{(2)}$ leaves the safe set. This again displays the effectiveness of the safe  ROAs obtained by our approach in certifying safe attraction.  
 \begin{figure}
     \centering
     \includegraphics[width=0.9\linewidth]{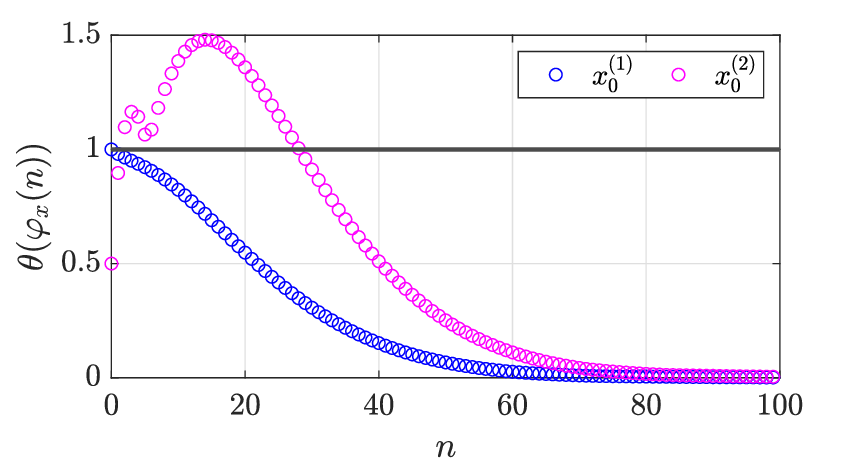}
     \caption{Profile of  $\theta$ along the trajectories starting from $x_{0}^{(1)}$ and $x_{0}^{(2)}$.}
     \label{fig:Theta_profile}
 \end{figure}

\section{Conclusion}
\label{sec:Conclusion}
In this paper, we proposed an iterative approach to under-estimate safe DOAs for general discrete-time autonomous nonlinear systems using implicit representations of backward reachable sets. The sets resulting from our iterative approach are monotonic, with respect to set inclusion, and are themselves safe regions of attraction,  with sublevel set representations, which are efficient for pointwise inclusion verification.

In future work, we aim to extend/adapt this framework to study robust domains of  attraction and domains of  null-controllability for perturbed and controlled discrete-time systems, respectively, which typically necessitate solving the computationally challenging Bellman-type equations \cite{xue2020characterization}.

\bibliographystyle{ieeetr}

\end{document}